\documentclass[a4paper]{article}

\usepackage[utf8]{inputenc}
\usepackage{lmodern}
\usepackage{amssymb,amsfonts,amsthm,amsmath,bbm,mathrsfs,aliascnt,mathtools,dsfont} 
\usepackage[english]{babel}
\usepackage{tikz}
\usetikzlibrary{angles,quotes}

\usepackage[colorlinks]{hyperref}
\usepackage{tikz}
\usetikzlibrary{decorations}
\usetikzlibrary{decorations.pathreplacing}
\tikzset{individu/.style={draw,thick}}

\theoremstyle{plain}
\newtheorem{theorem}{Theorem}
\newtheorem{corollary}[theorem]{Corollary}
\newtheorem{lemma}[theorem]{Lemma}
\newtheorem{proposition}[theorem]{Proposition}

\theoremstyle{definition}

\theoremstyle{remark}
\newtheorem{remark}[theorem]{Remark}

\numberwithin{equation}{section}

\usepackage{tikz}

\makeatletter \newcommand \listoftodos{\section*{Todo list} \@starttoc{tdo}}
\newcommand\l@todo[2]
{\par\noindent \textit{#2}, \parbox{10cm}{#1}\par} \makeatother

\newcommand{\N}{\mathbb{N}}

\newcommand{\R}{\mathbb{R}}
\newcommand{\C}{\mathbb{C}}

\newcommand{\Z}{\mathcal{Z}}

\newcommand{\indset}[1]{\mathds{1}_{#1}}

\renewcommand{\bar}[1]{\overline{#1}}

\newcommand{\blambda}{\boldsymbol{\lambda}}
\newcommand{\e}{\mathrm{e}}
\newcommand{\dd}{\mathrm{d}}

\renewcommand{\Sigma}{\boldsymbol{\sigma}}

\DeclareMathOperator{\E}{\mathbb{E}}

\renewcommand{\P}{\mathbb{P}}

\newcommand{\m}{\mathbf{m}}
\renewcommand{\epsilon}{\varepsilon}

\title{Another  solvable population dynamic\\ with uniform ancestral memory}
\author{Jean Bertoin\thanks{Institute of Mathematics, University of Zurich, Switzerland.} }
\date{ }

\begin{document}

\maketitle

\begin{abstract} We study a population model with non-overlapping generations, where individuals have a type and retain the memory of the types of all their ancestors. We assume that for each individual, the reproduction law depends on the empirical distribution of ancestral types, unlike for a classical multi-type branching process where it hinges on the parental  type only. We further assume the mean reproduction matrix has rank one.
Building on techniques from our previous works \cite{BM1, BM3}  with Bastien Mallein, which themselves draw on foundational contributions \cite{FGP, FDP} by Philippe Flajolet and collaborators in the context of analytic urns, we analyze the population at large generations. In particular, we determine the Malthusian parameter and estimate the expected size of the sub-population whose empirical distribution of ancestral types lies within a given set.

\end{abstract}


\medskip


\section{Introduction} 
Our motivation is to investigate to what extent making the reproduction of individuals depend on ancestral lineages 
can impact growth rates in dynamics of populations. For this, we have designed an elementary model which is mathematically tractable. This enables us to  reveal some rather unexpected features that may counterbalance some otherwise overly simplistic aspects. We now first present the model without memory, which will serve as a benchmark,
and then introduce its modified version with ancestral memory. In a second step we will further specify key elements of the evolution  to make the model analytically solvable.

The benchmark population model without memory  is simply the multi-type branching process with a finite  space of types $S$, as considered for instance in \cite[Chapter V]{AN}. 
Its evolution is entirely encoded by the family  of reproduction laws $(p(s, \cdot))_{s\in S}$ on $\N^S$. 
Specifically, for each $s\in S$
and each  $\mathbf{n}=(n(s'))_{s'\in S}\in \N^S$, $p(s,\mathbf{n})$
is the probability that a parent of type $s$ has exactly $n(s')$ children of type $s'$ for every $s'\in S$. 
It is well-known that many properties of the multi-type branching process hinge on the  mean reproduction matrix $R\in \R_+^{S\times S}$,
whose entries $R(s,s')$ give the mean number of children of type $s'$ for a parent of type $s$; see \cite[Section V.2]{AN}.

We next turn our attention to the dynamics with ancestral memory. For any individual at generation $n\geq 0$, say $\mathfrak{z}$, we consider the empirical distribution of its ancestral types,
\begin{equation} \label{E:empmeas}
L_{\mathfrak{z}}=\frac{1}{n+1}\sum_{j=0}^n \delta_{s_j},
\end{equation}
where  $(s_0, \ldots, s_n)$ is the sequence of types of its forebears. That is,  $s_n$ is the type of $\mathfrak{z}$, $s_{n-1}$ that of its parent, ..., and $s_0$ that of its first ancestor. Imagine that when $\mathfrak{z}$ reproduces, a type is activated at random\footnote{A related population model with ancestral memory, where the activated type is that of the forebear at some random past generation, has been also considered recently by the author in \cite{BerBer}. However the latter is not explicitly solvable.}
  according to  $L_{\mathfrak{z}}$. In other words, we may think that a type is sampled uniformly from an urn with
 composition given by  the ancestral types of $\mathfrak{z}$.
Then for any  $\mathbf{n}=(n(s'))_{s'\in S}\in \N^S$, 
the conditional probability given the activated type is $s$,  that $\mathfrak{z}$ has exactly $n(s')$ children of type $s'$ for every $s'\in S$, equals $p(s,\mathbf{n})$.

Of course, such random dynamics where the frequency of an event is proportional to its past occurrence
 have been widely considered in the literature,  notably in models with preferential attachment, with linear reinforcement, and of self-interacting Markov chains (see \cite{BWZ} for many precise references). 
It should be understood that, except for a few degenerate cases, letting reproduction depend on the ancestral memory undermines essential features of branching processes. The 
 future evolution given the family of types of individuals observed at a present generation is no longer  conditionally  independent of the past,
and at this level of generality, even the most basic questions about the asymptotic behavior of population models with ancestral memory seem challenging to analyze quantitatively. 

Nonetheless, a special case, called reinforced Galton--Watson process, has been effectively studied in a series of joint works \cite{BM1, BM2, BM3} with Bastien Mallein. There, the type
of an individual is simply the size of its sibling --hence $S$ is a finite set of positive integers-- and the entries of the mean reproduction matrix  have the specific form
\[R(s,s')= q s \indset{s=s'} + (1-q) s'\nu(s'),\] where $0\leq q \leq 1$ is a reinforcement parameter and $\boldsymbol{\nu}=(\nu(s'))_{s'\in S}$ a sub-probability measure on $S$. 
In words, when the type $s$ is activated for reproduction, with probability $q$ the parent begets exactly $s$ children  (thus each of type $s$), and with complementary probability $1-q$, an independent number of children is produced according to  $\boldsymbol{\nu}$\footnote{Allowing $\boldsymbol{\nu}(S)<1$  enables having no children with positive probability,  $\nu(0)=1-\boldsymbol{\nu}(S)$.}.
We were able to adapt to this setting deep ideas of Philippe Flajolet and co-authors \cite{FGP, FDP} on analytical urns, which combine a variety of tools including  notably transport equations and the analysis of singularities of generating functions. 

The  present work chiefly relies on the observation that the same approach applies as well in another situation, namely when the mean reproduction  matrix  has rank one.
More precisely, we consider   a nonnegative  function $\m\in [0,\infty)^S$ that we often view as a column vector with nonnegative entries $m(s)$, and implicitly exclude the trivial case when $\m$ is identically $0$. Similarly, we consider a probability measure $\Sigma$ on $S$ with full support that we often view  as a row vector with  entries $\sigma(s)>0$ and  $\sum \sigma(s)=1$. We use standard matrix product notation and assume henceforth that the mean reproduction matrix has the form 
\[ R = \m \Sigma. \]
In words,   the mean number of children of type $s'$ when the type $s$ has been activated for reproduction is simply the product $R(s,s')=m(s)\sigma(s')$.
So the activated type determines the mean size of the offspring; however the repartition of the types of the offspring is always $\Sigma$, independently of the activated type. 

We denote  the population with ancestral memory at generation $n$ by $\Z_n$ and its size (number of individuals) by $|\Z_n|$. We write   $\P_s$ for the probability measure induced by the dynamics with ancestral memory when  the initial population is reduced to a single ancestor  of type $s$, and simply $\P=\P_{\Sigma}$ when the type of the sole ancestor is random and distributed according to $\Sigma$.
A special case of our main result determines the behavior of the mean size of the population at a large generation; it can be stated as follows.

For every $s\in S$ with $m(s)\neq 0$, we have
\begin{equation}\label{E:T1}
\E_s(|\Z_n|) \sim c(s)n^{m(s)/m_* -1} M^{n} , \qquad \text{ as }n\to \infty, 
\end{equation}
where\footnote{Recall that we use the matrix product notation and that for every $x>0$, the function   $x^{\m+1}: s\mapsto x^{m(s)+1}$ on $S$ is viewed as a column vector. Thus, in \eqref{E:ms*}, one has to read read 
\[ \Sigma(x^{\m+1})= \sum_{s\in S} \sigma(s) x^{m(s)+1}. \]
}
  \begin{equation} \label{E:ms*}
m_* = \max_S \m
 \quad, \quad \frac{1}{M} = \int_1^{\infty} \frac{\dd x}{\Sigma(x^{\m+1}) },
\end{equation}
and $c(s)$ is some explicit positive finite constant depending on the type $s$. In words,  $M$  governs the exponential growth (or decay) of the mean population size, and hence plays the role of the Malthusian parameter.
Interestingly, there is a rational correction with negative exponent $m(s)/m_* -1$ to the exponential growth, except when $m(s)=m_*$; a similar phenomenon was observed for reinforced Galton--Watson processes in \cite{BM1}.

It is interesting to point out that, except in the degenerate case when the function $\m$ is constant, Jensen's inequality yields $\Sigma(x^{\m+1})>x^{\Sigma \m +1}$ for all $x>1$, and  
hence
  \begin{equation} \label{E:rmu}\Sigma \m < M.
\end{equation}
The interpretation of \eqref{E:rmu}  is that the Malthusian parameter $\Sigma \m$ of  the multi-type branching process\footnote{Plainly, for the branching process without memory, if the initial population is reduced to a single ancestor of type $s$, then the mean number of individuals of type $s'$ at generation $n\geq 1$ equals
$ m(s)(\Sigma \m)^{n-1}\sigma(s') $, where $\Sigma \m>0$ is the unique positive eigenvalue of $R=\m \Sigma$.}
 is strictly  smaller than that for the dynamics with ancestral memory. Loosely speaking,  the latter is a reinforced version  of the former,
where reinforcement means that the likelihood of previously observed reproduction events is higher; see the survey \cite{Pem}. As a consequence, if  an individual $\mathfrak{z}$ has a large proportion of forebears with prolific types (meaning with $m(s)$ large), then in turn  $\mathfrak{z}$ is more likely to have a large offspring. This creates  a feedback loop that 
reinforces the prolificness of the population model, as \eqref{E:rmu} confirms.

Our main result, Theorem \ref{TM}, encompasses  \eqref{E:T1} as a special case. It  will be stated and established in Section 2. Its proof is split into two main steps. The first relates the dynamics with ancestral memory to a  Yule process for which iid marks are assigned to all individuals. The main outcome of this relation, Theorem \ref{T2}, identifies the generating function of a sequence of moments. The second step follows the strategy of Flajolet \textit{et al.}
We check that  this generating function can be extended analytically  to a so-called $\Delta$-domain, and  that its singularities can be analyzed. The proof of Theorem \ref{TM} is  completed  from a transfer theorem. 

Finally, 
Section 3 offers some applications of large deviations techniques to the finer study of the empirical distribution of ancestral types of individuals, somehow in the same direction as
in the study of empirical measures of reinforced Markov
chains
 in \cite{BW}. 
Our main output there, Theorem \ref{T3},  strengthens \eqref{E:rmu} considerably. Informally speaking, it shows that as $n\to \infty$, the mean number of individuals $\mathfrak{z}$ at generation $n$ such that $L_{\mathfrak{z}}$ belongs to any given subset of probability measures, is always larger for the dynamics with ancestral memory than for the benchmark model of multi-type branching process. 

\section{Main result}

Stating the main result of this work requires further notation. Fix some function $\blambda\in \R^S$.
For any individual $\mathfrak{z}\in \Z_n$ in the population at generation $n$, 
we set
\[\mathcal E_{\blambda}(\mathfrak{z})= \exp\left( (n+1)  L_{\mathfrak{z}} \blambda\right)= \exp\left( \sum_{j=0}^n \lambda(s_j)\right)  , \]
where $L_{\mathfrak{z}}$  is the empirical distribution of the ancestral types, see \eqref{E:empmeas}. 
We are interested in estimating the mean 
\begin{equation} \label{E:quantint}
\E_s\left( \sum_{\mathfrak{z}\in \Z_n} \mathcal E_{\blambda}(\mathfrak{z})\right) 
\end{equation}
for a  large generation $n$. 

For any $x>0$ , we view as usual the function on $S$,
\[\e^{\blambda}x^{\m+1}: s\mapsto \e^{\lambda(s)} x^{m(s)+1},\]
as a column vector, and write
\begin{equation} \label{E:a}
\frac{1}{a(x)}=\Sigma(\e^{\blambda}x^{\m+1})=\sum_{s\in S} \sigma(s)\e^{\lambda(s)}  x^{m(s)+1} 
\end{equation}
and define
\begin{equation} \label{E:A}
  A(x) = \int_1^x a(y) \dd y.
\end{equation}

\begin{theorem} \label{TM} When $n\to \infty$, we have for every $s\in S$ with $m(s)\neq 0$  that 
\[\E_s\left( \sum_{\mathfrak{z}\in \Z_n} \mathcal E_{\blambda}(\mathfrak{z})\right) 
 \sim  c(s, \blambda)\, n^{-1+m(s)/m_*} A(\infty)^{-n} , \ \]
where 
\[c(s, \blambda) = \left( c_* A(\infty)\right)^{-m(s)/m_*} \frac{\e^{\lambda(s)}}{\Gamma(m(s)/m_*)} \,\]
and  $c_*$ is  given in \eqref{E:c_*} below.
\end{theorem}
Of course,  \eqref{E:T1} follows from Theorem \ref{TM}  by taking $\blambda \equiv 0$ and $M=1/A(\infty)$.

\subsection{Connection to a Yule process with iid marks}

We start by making the quantity of interest \eqref{E:quantint} more explicit. In this direction we introduce a sequence $(\xi_k)_{k\geq 0}$ of i.i.d. variables on $S$ with law $\Sigma$,
and  an independent  sequence $(u_\ell)_{\ell\geq 1}$, where each variable $u_\ell$ has the uniform distribution on $\{0, \ldots, \ell-1\}$ and is further independent of the other
$u_n$'s.
It is convenient to make  a slight abuse of notation and write $\P$ for the probability measure induced by these two sequences and $\P_s$ for its conditional version 
given $\xi_0=s$; so $\P= \P_{\Sigma}$ in the usual notation.

\begin{lemma} \label{L1} For every $n\geq1$ and $s\in S$, there is the identity
\[ \E_s\left( \sum_{\mathfrak{z}\in \Z_n} \mathcal E_{\blambda}(\mathfrak{z})\right)  = \e^{\lambda(s)} \E_s\left( \prod_{\ell=1}^n \e^{\lambda(\xi_\ell)} m(\xi_{u_\ell})\right).\]
\end{lemma}
\begin{proof}
Consider any sequence $s_1, \ldots, s_n$ in $S$ and any sequence $v_1, \ldots, v_n$ of integers with $0\leq v_\ell < \ell$ for $\ell=1, \ldots, n$.
Let $\Z_n(s_1, \ldots, s_n; v_1, \ldots, v_n)$ denote the set of individuals in the population at generation $n$ such that
the sequence of types along its ancestral lineage is $(s,s_1, \ldots, s_n)$, and further, for every $\ell=1, \ldots, n$,  $v_{\ell}$ is the generation that has been selected to beget the forebear at  generation $\ell$. On the one hand,  for every $\mathfrak{z}\in \Z_n(s_1, \ldots, s_n; v_1, \ldots, v_n)$, one has by definition that
\[\mathcal E_{\blambda}(\mathfrak{z})= \e^{\lambda(s)} \prod_{\ell=1}^n \e^{\lambda(s_\ell)}.\]
On the other hand, an iteration over generations enables us to translate rigorously the verbal description of the population dynamics with ancestral memory into
\[ \E_s(|\Z_n(s_1, \ldots, s_n; v_1, \ldots, v_n)|)= \frac{1}{n!} \prod_{\ell=1}^n \sigma(s_\ell)m(s_{v_\ell}).\]
The  summation over all sequences  $s_1, \ldots, s_n$  and  $v_1, \ldots, v_n$ yields our claim. 
\end{proof} 

Lemma \ref{L1} points at a connection to a Yule process with iid marks. 
Indeed, consider first the  infinite random recursive tree $\mathcal T$ on $\N=\{0,1, \dots\}$  rooted at $0$ and with oriented edges  $(u_{\ell},\ell)$ for  $\ell \geq 1$.
It is well known that $\mathcal T$ describes the genealogy of a standard Yule process where 
individuals  are enumerated in the order of their birth-time,
so  the root vertex $0$ represents the ancestor, $\ell\geq 1$ the $\ell$-th born individual, and $u_{\ell}$ the parent of $\ell$.  
In other words, $\mathcal T$ is the genealogical tree of the entire Yule population, where directed edges connect parents to children.

We next further assign the mark $\xi_{n}$ to every vertex $n\geq 0$; the interpretation being that individuals in the Yule process receive i.i.d. marks (or types) distributed according to $\Sigma$, independently of the genealogy. 
We write $Y_t$ for the size of the Yule population at time $t\geq 0$. The process $(Y_t)_{t\geq 0}$ is independent of the genealogical tree $T$ and its marks, and since $Y_t$ has the geometric distribution with parameter $\e^{-t}$, Lemma \ref{L1} immediately yields:

\begin{corollary} \label{C1} For every $t\geq 0$ and $s\in S$, we have the identity
\[  \e^{-t} \sum_{n=0}^{\infty} (1-\e^{-t})^n  \E_s\left( \sum_{\mathfrak{z}\in \Z_n} \mathcal E_{\blambda}(\mathfrak{z})\right)  =  \e^{\lambda(s)} \E_s\left( \prod_{\ell =1}^{Y_t-1} \e^{\lambda(\xi_\ell)} m(\xi_{u_\ell})\right) ,\]
where these quantities can be finite or infinite. 
\end{corollary} 

We shall now study the right-hand side of the identity in Corollary \ref{C1}, which we simply denote by 
\[ \Phi_s(t) =\e^{\lambda(s)} \E_s\left( \prod_{\ell =1}^{Y_t-1} \e^{\lambda(\xi_\ell)} m(\xi_{u_\ell})\right) .\]
The next key step of our analysis is the observation that 
each function $\Phi_s$ has a simple expression in terms 
of the linear combination
\begin{equation} \label{E:PhiSigma}
\Phi(t) = \sum_{s\in S} \sigma(s) \Phi_s(t) = \E\left(\e^{\lambda(\xi_0)} \prod_{\ell =1}^{Y_t-1} \e^{\lambda(\xi_\ell)} m(\xi_{u_\ell})\right).
\end{equation}

\begin{lemma} \label{L2} For every $s\in S$ and all $t\geq 0$, there is the identity
\[\Phi_s(t) = \exp\left(\lambda(s) + m(s)\int_0^t \Phi(x)\dd x -t\right).\]
\end{lemma}

We propose two proofs of Lemma \ref{L2},  which both  rely on the branching property of the Yule process, but from two different view points.
The first is actually a partial proof only; it 
argues that the function $\Phi_s$ satisfies a kind of Lotka-Volterra differential equation which stems from the branching property applied at the first birth event.
 This approach is  standard, but has the drawback that its validity is limited to times $t$ before  
 the explosion time 
\[ T^*=\sup_{x>0} \inf\{t\geq 0: \Phi(t)>x\},\]
which will be determined in the forthcoming Corollary \ref{C2}. 
The second  proof relies on  the fact that  the times when an individual begets children in a Yule process are described by a Poisson process, and Poissonian calculus.
 This argument applies no matter whether explosion has already occurred or not.
\begin{proof}[Partial proof of Lemma \ref{L2} by differential calculus] To start with, work under $\P_s$ conditionally on the event that the first birth  in the Yule process occurs at time $t'<t$ and that the first child receives the mark $s'$.
We distinguish the next children born after time $t'$, depending on whether they descend from the first child or not.
Since  $\xi_{u_1}=s$, 
the branching property of the marked Yule process yields the decomposition
\[ \e^{\lambda(s)} \prod_{\ell =1}^{Y_t-1} \e^{\lambda(\xi_\ell)} m(\xi_{u_\ell}) = m(s) \eta \eta',\]
where $\eta$ and $\eta'$ are independent, and $\eta$ (respectively, $\eta'$) is distributed as 
\[ \e^{\lambda(\xi_0)} \prod_{\ell =1}^{Y_{t-t'}-1} \e^{\lambda(\xi_\ell)} m(\xi_{u_\ell}) \]
under $\P_s$ (respectively, under $\P_{s'}$). 

Now work again under the unconditional law $\P_s$. Since   the birth of the first child occurs at an exponentially distributed time
and its mark is independent and distributed according to $\Sigma$, the decomposition above translates into 
\[ \Phi_s(t) = \e^{-t}\e^{\lambda(s)} + m(s) \int_0^t \e^{-t'} \Phi_s(t-t') \Phi(t-t') \dd t',\]
Changing the variables $t-t'=u$ in the integral gives 
\[ \e^t\Phi_s(t) = \e^{\lambda(s)}+ m(s) \int_0^t \e^u\Phi_s(u) \Phi(u) \dd u ,\]
and we arrive at a generalized Lotka--Volterra equation for $t<T^*$
\[ \Phi_s'(t) =  \Phi_s(t)(m(s) \Phi(t)-1)\]
with boundary condition $\Phi_s(0)= \e^{\lambda(s)}$. Its unique solution is that given in the statement. 
\end{proof}

\begin{proof}[Proof of Lemma \ref{L2} by Poisson calculus] We work under $\P_s$ and 
recall that children in the Yule process receive at birth an independent mark distributed according to $\Sigma$. 
So the point process that records the birth-times and the marks
of the children of the ancestor is described by a Poisson point process on $\R_+\times S$ with intensity $\dd t \otimes \Sigma$. 
By the branching property, conditionally on the latter, each child, say born with mark $s'$,  
generates in turn an independent marked Yule process  distributed according to $\P_{s'}$.

Each child of the ancestor, say born at time $t'<t$ and with mark $s'$,  thus contributes together with its descent to the product 
\[\prod_{\ell =1}^{Y_t-1} \e^{\lambda(\xi_\ell)} m(\xi_{u_\ell}) \]
 by a factor which has the same law as 
 \begin{equation} \label{E:decompos}
 m(s)\e^{\lambda(\xi_0)}\prod_{\ell =1}^{Y_{t-t'}-1} \e^{\lambda(\xi_\ell)} m(\xi_{u_\ell})
 \end{equation}
  under $\P_{s'}$, and as a consequence, has expectation 
$m(s)\Phi_{s'}(t-t')$.
By the Laplace formula for Poisson measures, this yields the identity
\begin{align*}
 \Phi_s(t) &= \exp\left(\lambda(s) +\sum_{s'\in S} \sigma(s')\int_0^t (m(s) \Phi_{s'}(t-t')-1)\dd t' \right) \\
 &= \exp\left(\lambda(s)+ m(s) \int_0^t \Phi(x)\dd x - t\right).
 \end{align*}
We stress that our application of the Laplace formula is legit  even though the quantities 
\eqref{E:decompos} are not bounded, because the intensity of the Poisson measure is finite.
\end{proof}
We can now determine when and where  explosion occurs. Recall the notation  \eqref{E:A}, and also define
 \[ \alpha= \sup\{x>1: A(x)<1\}.\]
In words,  if $A(\infty)\leq 1$, then $ \alpha= \infty$;  otherwise, if $A(\infty)>1$, then  $ \alpha\in (1, \infty)$ is the unique solution to $A(\alpha)=1$. 
We further introduce  the function 
\begin{equation} \label{E:bij} \Psi(t)=  \exp\left( \int_0^t \Phi(x)\dd x\right), \qquad t\geq 0.\end{equation}

\begin{corollary} \label{C2} There is the identity
\[ \alpha =\Psi(T^*) , \]
and we can distinguish the following three cases:
\begin{itemize}  
\item[(i)] (Supercritical case) If $A(\infty)<1$, then the explosion time  is finite and given by 
\[T^* = -\log(1-A(\infty)).\]
We have also   $\Phi(t)= \infty$ for all $t\geq T^*$. 

\item[(ii)] (Critical case) If $A(\infty)=1$, then   $\alpha = T^*=\infty$. 
 
\item[(iii)] (Subcritical case)
 If $A(\infty)> 1$, then  $\alpha < \infty$ and $ T^*=\infty$. 
 
\end{itemize}
\end{corollary}

\begin{proof} To start with, note from Lemma \ref{L2} that for any $s\in S$  and $t\geq 0$, there are the equivalences
\[ \Phi_s(t)=\infty \Longleftrightarrow  \Psi(t)=\infty \Longleftrightarrow \int_0^{t} \Phi(x) \dd x=\infty,\]
and in that case, we have also $\Phi_{s}(t')=\infty$ for any $t'\geq t$. 
Next, fix  $c>0$  sufficiently large so that $\e^{\lambda(s)} m(s') \leq c$ for all $s,s'\in S$, and take any $t<T^*$. It is easy to check that we can choose $t'>t$ sufficiently close to $t$ such that
$\E\left(c^{Y_{t'}-Y_t}\right)<\infty$, and then, from \eqref{E:PhiSigma}, that $\Phi$ remains bounded on $[0,t']$. As a consequence, there are the equivalence
\[ t < T^* \Longleftrightarrow \Phi(t)<\infty, \]
and the implication
\begin{equation} \label{E:zetafin}   T^*<\infty \Longrightarrow 
\Phi(T^*) = \Psi(T^*)=\infty.
\end{equation} 

Next, by substitution in \eqref{E:PhiSigma}, we get from Lemma \ref{L2} that
\[ \Phi(t)= \sum_{s\in S} \sigma(s) \exp\left(\lambda(s) + m(s) \int_0^t \Phi(x)\dd x - t\right). \]
and then, using the notations \eqref{E:a}  and \eqref{E:bij}, we  rewrite the preceding identity  for $t<T^*$   as 
 \begin{equation} \label{E:changev}
\Psi'(t) a(\Psi(t))= \e^{-t}.
 \end{equation}
Plainly, $\Psi: [0,T^*]\to [1, \Psi(T^*)]$ is  bijective. So recalling \eqref{E:A}, integrating \eqref{E:changev},   and performing a change of variables
yield
\[ A(\Psi(t))=1-\e^{-t}, \qquad 0\leq t \leq T^*.\]

We focus on this identity for the boundary point $t=T^*$, and assume first that $T^*<\infty$, so $\Psi(T^*)=\infty$ by \eqref{E:zetafin}. We then see that $A(\infty)<1$ and hence  $\Psi(T^*)=\alpha$.
Next assume $T^*=\infty$ and $\Psi(\infty)=\infty$. 
The same argument gives $A(\infty)=1$  and $\Psi(T^*)=\alpha$. 
Finally assume $T^*=\infty$ and $\Psi(\infty)<\infty$. Then $A(\Psi(\infty))=1$, which is possible if and only  if $A(\infty)>1$, and then again $\Psi(T^*)=\alpha$. 
\end{proof}

We can now establish the main result of this section.

\begin{theorem} \label{T2} 
For every $s\in S$ and $0<y<A(\infty)$, we have
 \[ \sum_{n=0}^{\infty} y^n\E_s\left( \sum_{\mathfrak{z}\in \Z_n} \mathcal E_{\blambda}(\mathfrak{z})\right) = \e^{\lambda(s)} B(y)^{m(s)}, \]
 where $B:[0,A(\infty))\to [1, \infty)$ is the inverse bijection of $A$.

\end{theorem}
 
 \begin{proof}  We have seen  in the proof of Corollary \ref{C2} 
 that the function $\Psi: (0,T^*)\to (1,\alpha)$ in \eqref{E:bij}  is the  inverse of the  $\mathcal{C}^{\infty}$-diffeomorphism 
\[x\mapsto -\log(1-A(x)).\]
Then from Corollary \ref{C1}, Lemma \ref{L2}, and the definition \eqref{E:bij}, we get that
 for every $0\leq y \leq 1-\e^{-T^*}$, we have
 \[ \sum_{n=0}^{\infty} y^n \E_s\left( \sum_{\mathfrak{z}\in \Z_n} \mathcal E_{\blambda}(\mathfrak{z})\right)  = \e^{\lambda(s)} \Psi(-\log(1-y))^{m(s)} = \e^{\lambda(s)} B(y)^{m(s)}. \]
 If $A(\infty)\leq 1$, then we know that $A(\infty)=1-\e^{-T^*}$, so the formula in the statement is proved.
 If $A(\infty)> 1$, then $T^*=\infty$ and we have proved the formula for $y\in[0,1]$ only; however the extension to $y\in [0,A(\infty)]$ follows by uniqueness of analytic continuation.
  \end{proof} 

As a sanity check of Theorem \ref{T2}, consider the case when $S$ is a singleton. Then  $A(x)=\e^{-\lambda} m^{-1}(1-x^{-m})$ and $B(y)= (1-\e^{\lambda}my)^{-1/m}$. Theorem \ref{T2} gives
for $y<\e^{-\lambda}/m$,
\[ \sum_{n=0}^{\infty} y^n \E_s\left( \sum_{\mathfrak{z}\in \Z_n} \mathcal E_{\blambda}(\mathfrak{z})\right)  = \frac{\e^{\lambda}}{1-\e^{\lambda}my} =  \sum_{n=0}^{\infty}  y^n \e^{(n+1) \lambda }m^n.\]
We conclude that \eqref{E:quantint}
 equals $ \e^{(n+1) \lambda} m^n$, as we should expect. 

\subsection{Singularity analysis of a generating function} 
\label{sec:singularity}

 Theorem \ref{T2}  provides a remarkably simple expression for the generating function of the sequence of moments \eqref{E:quantint}
  in terms of the inverse function $B$ of $A$. As  amply demonstrated by Flajolet and co-authors, the analysis of singularities is a most efficient tool to extract sharp estimates of these moments
 from analytic properties of the generating function, and  the purpose of this section is to check that this methodology applies to our setting.
 
 We start by recalling  a few key notions in this area, and refer to \cite[Chapter VI]{FS} for a complete account. We say that a complex function $f$ 
 is  $\Delta$-analytic at a complex number $z_0$ if there exist  some $R>|z_0|$ and $\theta\in(0,\pi/2)$ such that $f$ is analytic in the domain
  \[ \Delta(\theta, R) = \{z\in \C: |z|< R, z\neq z_0, |\arg(z-z_0)|>\theta\}.\]
   
 \begin{proposition}\label{P1}  For every $s\in S$ with $m(s)>0$, there exists an extension $f$ of $B^{m(s)}$ which is  $\Delta$-analytic at $A(\infty)$. 
 Moreover,  as $z\in \Delta$ goes to $A(\infty)$, we have 
\[f(z) \sim \left( c_* (A(\infty)-z)\right)^{-m(s)/m_*},\]
where $m_*$ is given in \eqref{E:ms*} and 
\begin{equation} \label{E:c_*}c_*= m_*\sum_{s\in S} \sigma(s) \e^{\lambda(s)} \indset{m(s)=m_*}.\end{equation}
 \end{proposition}
  
  Theorem \ref{TM} follows immediately from Proposition \ref{P1} by transfer, as we now explain.

\begin{proof} [Proof of Theorem \ref{TM}] 
   For any $s\in S$ with $m(s)\neq 0$, we get by combining Theorem \ref{T2} and Proposition \ref{P1},   that as $z\in \Delta$ tends to $A(\infty)$, 
  \[ \sum_{n=0}^{\infty} z^n \E_s\left( \sum_{\mathfrak{z}\in \Z_n} \mathcal E_{\blambda}(\mathfrak{z})\right)  =\e^{\lambda(s)}  B(z)^{m(s)} \sim \e^{\lambda(s)}  \left( c_* (A(\infty)-z)\right)^{-m(s)/m_*} .\]
We rewrite the right-hand side as
\[ \e^{\lambda(s)}  \left( c_* A(\infty)\right)^{-m(s)/m_*} (1-z/A(\infty))^{-m(s)/m_*} \]
and deduce from the transfer theorem \cite[Corollary VI.1]{FS} that
 \[  \e^{-\lambda(s)}  A(\infty)^n \left( c_* A(\infty)\right)^{m(s)/m_*}  \E_s\left( \sum_{\mathfrak{z}\in \Z_n} \mathcal E_{\blambda}(\mathfrak{z})\right) \sim \frac{n^{-1+m(s)/m_*}}{\Gamma(m(s)/m_*)} ,
 \]
 which is Theorem \ref{TM}.
  \end{proof}
  
  The rest of this section is devoted to the proof of Proposition \ref{P1}, for which two analytic lemmas are needed.
  The purpose of the first is to extend $B$ near its critical point $A(\infty)$.  Recall the notation \eqref{E:ms*} and \eqref{E:c_*} and introduce the function 
  \[\bar A(x)= A(\infty)-A(x^{-1/m_*}), \qquad  x>0.\]  
  \begin{lemma} \label{L:extsing}
  There is some sufficiently small $r>0$ such that $\bar A$
  can be extended analytically to a slitted disk
  $D=\{z\in \C: |z|<r \text{ and } z\not\in \R_-\}$.
  Its derivative $\bar A'$ satisfies
  \[ \lim_{D\ni z\to 0} \bar A'(z)=1/c_*.\]
   \end{lemma}
   
\begin{proof} Set 
$\bar a(x) = a(x^{-1/m_*})$, so
 \[\bar A(x)= \frac{1}{m_*} \int_0^x  {\bar a(t)} t^{-1-1/m_*} \dd t. \]
 Then observe that the inverse of the integrand can be expressed as 
\[ \frac{ t^{1+1/m_*} }{\bar a(t)} = \sum_{s\in S} \sigma(s) \e^{\lambda(s)} t^{1-m(s)/m_*}.\]
The right-hand side defines an analytic function on the slitted complex plane $\C\backslash \R_-$,
which has limit $c_*/m_*$ at $0$. The statement follows readily.
\end{proof}
Next, it  is convenient to introduce 
  \[ \varphi(z) = \Sigma \left(\e^{z\m + \blambda}\right) = \sum_{s\in S} \sigma(s) \exp(\lambda(s) + m(s)z), \qquad z\in \C.  \]
  The relevance of $\varphi$ stems from the observation that the function of the real variable $x\mapsto A(\e^x)$  has derivative $1/\varphi$ and its inverse,
   $\ln B$, solves the autonomous differential equation 
  \[
   \left( \ln B \right)' = \varphi( \ln B), \qquad \ln B(0)=0. 
 \]
  The coefficient $\varphi$ is an entire function, so the  Cauchy-Kovalevskaya theorem ensures local existence and uniqueness of the solution and yields an analytic extension of $\ln B$.
  We still have to verify that the solution remains well-defined on a sufficiently large domain. For this purpose, we 
  work on a ray $\{t\e^{i\theta}: t\geq 0\}$ with angle $\theta\in(-\pi, \pi]$, consider the differential equation in the variable $t\geq 0$,
   \begin{equation}\label{E:autoed}
    g_{\theta}'(t)= \e^{i\theta} \varphi(g_{\theta}(t)), \qquad g_{\theta}(0)=0,
      \end{equation}
  and we write $\zeta(\theta)$ for the  explosion time. Plainly, for $\theta=0$, one has $g_0=\ln B$ and $\zeta(0)=A(\infty)$.
  
  \begin{lemma} \label{L:extbulk}
  For every $\epsilon\in (0, \pi/2)$, we have
  \[ \inf_{|\theta|>\epsilon} \zeta(\theta) > A(\infty).\]
  \end{lemma}
  \begin{proof} We check  there exist $0<t_1<t_2< A(\infty)$ such that, for any angle $\theta$ with $|\theta|>\epsilon$, 
  \[ |g_{\theta}(t_2)| < g_0(t_1).\]
  Indeed, writing $c=\Sigma \e^{\blambda}$, we have
  $g_{\theta}(t)\sim ct \e^{i\theta}$  as $t\to 0+$, uniformly in $\theta$. Pick $0<\eta<(1-\cos \epsilon)/2$ .
  By the triangular inequality and \eqref{E:autoed}, we have for all $t>0$ sufficiently small 
  \[ |g_{\theta}'(t) | \leq  \sum_{s\in S} \sigma(s) \e^{\lambda(s)} \exp(m(s) tc (\eta+\cos \epsilon )) \]
  and 
  \[ g_0'(t) \geq  \sum_{s\in S} \sigma(s) \e^{\lambda(s)} \exp(m(s) tc (1-\eta )).\]
  The existence of $t_1$ and $t_2$  as above follows readily.   
  
  We deduce that there is the inequality
  \[ |g_{\theta}(t_2+t)| < g_0(t_1+t), \qquad \text{ for all } t<A(\infty)-t_1.\]
  Indeed, if this failed, then there would exist a first time $t_3 > 0$ at which 
  \[|g_{\theta}(t_2+t_3)| =  g_0(t_1+t_3).\]
  This is absurd, because $  |g_{\theta}(t_2+t)| < g_0(t_1+t)$ for all $0\leq t < t_3$, and therefore,
   again by the triangular inequality, 
   \[|g'_{\theta}(t_2+t)| = |\varphi( g_{\theta}(t_2+t))| \leq \varphi(| g_{\theta}(t_2+t)|) \leq  \varphi( g_{0}(t_1+t)) =g_{0}'(t_1+t).\]

We conclude that
 \[\zeta(\theta) \geq \zeta(0)+t_2-t_1,\]
 which completes the proof.   
  \end{proof}
  
  We now have the two ingredients needed to establish Proposition \ref{P1}.
  \begin{proof}[Proof of Proposition \ref{P1}]
  The reader may find useful to look at the figure below which illustrates the 
three different domains on which $B^{m(s)}$ can be extended.
  Their union contains a $\Delta$-domain at $A(\infty)$.

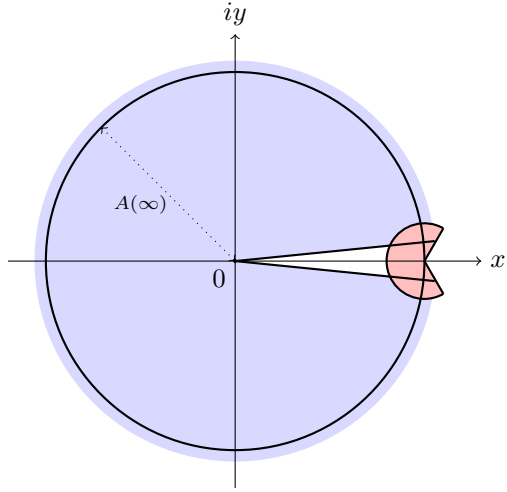
  \begin{tikzpicture}[scale=.5]

    \def\R{5}
    \def\Rout{5.3}
    \def\r{1}
    \def\eps{5.7296} 
    \def\alpha{60}   

    \fill[blue!15]
        (0,0)
        -- (\eps:\Rout)
        arc[start angle=\eps,end angle=180,radius=\Rout]
        arc[start angle=180,end angle=360-\eps,radius=\Rout]
        -- cycle;

    %
    \fill[red!25]
        (5,0)
        -- ++(\alpha:\r)
        arc[start angle=\alpha,end angle=180,radius=\r]
        arc[start angle=180,end angle=360-\alpha,radius=\r]
        -- cycle;

    \draw[thick] (0,0) circle (\R);

    \draw[thick]
        (0,0) -- (\eps:\Rout);

    \draw[thick]
        (0,0) -- (-\eps:\Rout);

    \draw[thick]
        (5,0) -- ++(\alpha:\r);

    \draw[thick]
        (5,0) -- ++(-\alpha:\r);

    \draw[thick]
        (5,0) ++(\alpha:\r)
        arc[start angle=\alpha,end angle=180,radius=\r]
        arc[start angle=180,end angle=360-\alpha,radius=\r];

    \fill (0,0) circle (1.5pt) node[below left] {$0$};

    \draw[->] (-6,0) -- (6.5,0) node[right] {$x$};
    \draw[->] (0,-6) -- (0,6) node[above] {$iy$};
    
  \def\radius{5}
  \def\angle{135} 

  \coordinate (A) at ({\radius*cos(\angle)}, {\radius*sin(\angle)});
  \coordinate (B) at (0,0);

  \draw[dotted, <->] (A) -- (B);
   \fill (-1.5,2)  node[below left] {$\ _{A(\infty)}$};
    
     \node[below, align=center] at (0,-6.5) {
     $B^{m(s)}$ is extended analytically to 
        the disk with radius $A(\infty)$ (by Theorem \ref{T2}), \\
      to  the red sector (by Lemma \ref{L:extsing})
     and  to  the blue sector (Lemma \ref{L:extbulk}).
    };

\end{tikzpicture}

  Recall Lemma \ref{L:extsing}.  Since $\bar A$ is analytic in the slitted disk $D$ with $\bar A(z)- z/c_*=o(z)$ as $z\in D$ goes to $0$, the following assertion can be deduced  from a standard argument based on Rouch\'e's theorem. 
There exist $r'\in(0,r)$ and  a sub-domain $D'\subset D$ with $D'\cap \R_+\neq \emptyset$,  such that the analytic extension of 
  $\bar A$ to $D'$ is a conformal map  to the sector 
  \[\Theta=\{z\in \C: 0<|z|<r' \text{ and }|\arg(z)|<2\pi/3\}.\]
  Let $\bar B$ denote the inverse (for the composition) of the latter, which maps $\Theta$ into the slitted disk $D$. We have for any $x>0$ small enough that
  \[ \bar B(x)^{-1/m_*} = B\left( A(\infty)-x\right),\]
  thus for any $s\in S$, $z\mapsto \bar B(A(\infty)-y)^{-m(s)/m_*}$ is the  analytic extension of $f=B^{m(s)}$ to the sector $A(\infty)-\Theta$. 
  Since $\bar B(z)\sim c_*z$ as $z\in \Theta$ goes to $0$ (recall that $\bar A(z)\sim z/c_*$),
  we have $f(z) \sim (c_*(A(\infty)-z))^{-m(s)/m_*}$ as $z$ approaches the critical point $A(\infty)$ in the sector $A(\infty)-\Theta$. 
 
 We next set $\epsilon = \arg( A(\infty)+ r' \e^{i\pi/3})$ and  recall Lemma \ref{L:extbulk}.
 We can choose $R$ with $A(\infty)<R<A(\infty)+r'$, such that $R< \zeta(\theta)$ for all angles $-\pi< \theta\leq \pi$ with $|\theta|>\epsilon$.
  The discussion preceding Lemma \ref{L:extbulk} shows that $\ln B$ can be extended analytically to the domain
 \[ \{z\in \C: 0<|z|< R \text{ and } |\arg(z)| > \epsilon\},\]
 and therefore the same holds for $B^{m(s)}$. Since we already observed that $B^{m(s)}$ extends analytically to the open disk with radius $A(\infty)$ centered at $0$, 
 by combining the extensions together, we conclude that $B^{m(s)}$ has an analytic extension to $\Delta( \pi/3, R)$, 
which completes the proof. 
  \end{proof}

 \section{Some applications of large deviations techniques}
 
A weaker version of Theorem \ref{TM} is that for any $s\in S$ with $m(s)>0$  and for any function $\blambda\in \R^S$, 
\begin{equation} \label{E:GEcond}
 \lim_{n\to \infty}
\frac{1}{n} \log \E_s\left( \exp\left( (n+1) \sum_{\mathfrak{z}\in \Z_n}  L_{\mathfrak{z}}  \blambda\right) 
\right)=  \Lambda(\blambda).
\end{equation}
Here, $L_{\mathfrak{z}}$ is the empirical distribution of ancestral types defined in \eqref{E:empmeas},  \[  \Lambda(\blambda)  = - \log \left( \int_0^{\infty} \frac{\dd y}{\Sigma \exp(\blambda+ y\m)}\right), \]
and, in the matrix product notation,
\[\Sigma \exp(\blambda+ y\m)={\sum_{s\in S}\sigma(s) \exp(\lambda(s) + y m(s))}.\]

Obviously, \eqref{E:GEcond} invites to apply the G\"artner--Ellis theorem, cf.  \cite[Theorem 2.3.6]{DZ}.  In this direction, we first gather properties of $\Lambda$ and of its Fenchel--Legendre transform $\Lambda^*$, which is 
 defined on the space $\mathcal{P}_S$ of probability measures on $S$ by
\[ \Lambda^*(\boldsymbol{\rho}) = \sup\{ \boldsymbol{\rho} \blambda -\Lambda( \blambda):  \blambda\in \R^S\} \in (-\infty,\infty].\]

\begin{lemma} \label{L:elemprop} The following assertions hold:
\begin{enumerate}
\item[(i)]
The function $\Lambda$ is $\mathcal{C}^{\infty}$ and convex on $\R^S$. Its gradient  $\nabla \Lambda$, viewed as a row vector,  takes values in the subspace  of probability measures on $S$ with full support.
\item[(ii)] The Fenchel--Legendre transform $\Lambda^* $ is a strictly convex function. 

\item[(iii)] For any $\boldsymbol{\rho}\in \mathcal{P}_S$ with full support, there exists $\blambda_{\boldsymbol{\rho}}\in \R^S$ such that
\[\nabla \Lambda(\blambda_{\boldsymbol{\rho}})=\boldsymbol{\rho} \quad \text{and} \quad \Lambda^*(\boldsymbol{\rho} )=\boldsymbol{\rho}  \blambda_{\boldsymbol{\rho}} - \Lambda(\blambda_{\boldsymbol{\rho}}). \]
Moreover, for any $\boldsymbol{\rho}'\in \mathcal{P}_S$ different from $\boldsymbol{\rho}$, there is the strict inequality
\[ \Lambda^*(\boldsymbol{\rho}' )- \boldsymbol{\rho}'  \blambda_{\boldsymbol{\rho}}  > \Lambda^*(\boldsymbol{\rho} )-  \boldsymbol{\rho} \blambda_{\boldsymbol{\rho}} . \]
\end{enumerate}
\end{lemma}

The proofs of  these claims are elementary; we refer to \cite[Section 3.2]{BM3} where similar verifications are made.
 
 We next derive an annealed large deviation principle for the empirical distributions of ancestral types.

\begin{corollary} \label{CGE} 
For any $s\in S$ with $m(s)\neq 0$, we have: 
\begin{enumerate}
\item[(i)]
 For every closed set $F\subset \mathcal P_S$, we have
 \[ \limsup_{n\to \infty} \frac{1}{n} \log  \E_s \left( \sum_{\mathfrak{z}\in \Z_n} \indset{F}(L_{\mathfrak{z}})\right) \leq -\inf_{\boldsymbol{\rho} \in F}   \Lambda^*(\boldsymbol{\rho}).\]

 \item[(ii)]
 For every open set $G\subset \mathcal P_S$, 
we have
 \[ \liminf_{n\to \infty} \frac{1}{n} \log  \E_s \left( \sum_{\mathfrak{z}\in \Z_n} \indset{G}(L_{\mathfrak{z}})\right) \geq -\inf_{\boldsymbol{\rho} \in G}   \Lambda^*(\boldsymbol{\rho}).\]

\end{enumerate}

\end{corollary}

\begin{proof} Consider for  every generation $n\geq 0$, a random probability measure  $X_n$ on $S$ with law given by
 \[ \E_s(f(X_n)) = \E_s\left(\sum_{\mathfrak{z}\in \Z_n} f(L_{\mathfrak{z}})\right) \big / \E_s(|\Z_n|) ,\]
 where $f:\mathcal P_S\to \R_+$ stands for a generic measurable function. We observe that 
  the sequence of the distributions of $X_n$ satisfies the large deviation principle under $\P_s$ for any $s\in S$ with $m(s)\neq 0$,  with the good rate function 
\[I=\Lambda^*+\Lambda(0).\]
Indeed, thanks to \eqref{E:GEcond}, we have for every $\blambda \in \R^S$ that 
\[ \lim_{n\to \infty} \frac{1}{n} \log \E_s \left( \exp( n  X_n \blambda)\right) = \Lambda(\blambda) - \Lambda(0).\]
The claim thus follows directly from the G\"artner-Ellis theorem and Lemma \ref{L:elemprop}.  
 To conclude, it suffices to write for any $E\subset \mathcal P_S$ that
\[ \E_s \left( \sum_{\mathfrak{z}\in \Z_n} \indset{E}(L_{\mathfrak{z}})\right) = \P_s(X_n\in E)  \E_s(|\Z_n|), \]
next to apply  \eqref{E:T1}, and finally to 
recall that $
\Lambda(0)= \log M, $
in the notation \eqref{E:ms*}. \end{proof}

\begin{remark} As a consequence of  Lemma \ref{L:elemprop} and Corollary \ref{CGE},  the empirical distribution of ancestral types of individuals  in the population model with ancestral memory concentrates exponentially fast around 
$\nabla \Lambda(0)$. Specifically,  for any closed subset $F \subset \mathcal{P}_S$ with $\nabla \Lambda(0)\not\in F$, there is some $c(F)<1$ such that for all $n$ sufficiently large, we have  the bound 
\[ \E_s \left( \sum_{\mathfrak{z}\in \Z_n} \indset{F}(L_{\mathfrak{z}})\right) \leq c(F)^n \E_s (|\Z_n|).\]
We further note from an integration by parts that
 $M=\nabla \Lambda(0)\m$,
 as we should expect.
\end{remark}

The chief purpose of this section is to compare Corollary \ref{CGE} with its counterpart for the benchmark model without memory that we now discuss.
In this direction, write $\mathcal B_n$ for the population at generation $n$ of a multi-type branching process with mean reproduction matrix $R=\m\Sigma$. Recall that $\Sigma \m>0$ is the (unique) positive eigenvalue of $R$ with left-eigenvector $\Sigma$.
We also introduce for every $\boldsymbol{\rho}\in \mathcal{P}_S$, 
\[J(\boldsymbol{\rho})=  \sum_{s\in S} \rho(s) \log\left(\frac{\rho(s)}{\sigma(s)m(s)}\right) =  D_\mathrm{KL}(\boldsymbol{\rho}\| \Sigma_{\m}) - \log (\Sigma \m ),\]
where $D_\mathrm{KL}$ denotes the Kullback--Leibler divergence and $\Sigma_{\m}$
 the $\m$-biased version of $\Sigma$ which is defined by
\[ \Sigma_{\m}(s)= \frac{\sigma(s)m(s)}{\Sigma \m}, \qquad s\in S.\]

The following statement  mirrors Corollary \ref{CGE}. It is merely a minor variation of Sanov's large deviation principle, and can be readily checked from the G\"artner--Ellis theorem. Details of the proof are left to the reader.

\begin{lemma} \label{L:Sanov}
For any $s\in S$ with $m(s)\neq 0$, we have:
\begin{enumerate}
\item[(i)]
 For every closed set $F\subset \mathcal P_S$, we have
 \[ \limsup_{n\to \infty} \frac{1}{n} \log  \E_s \left( \sum_{\mathfrak{z}\in \mathcal B_n} \indset{F}(L_{\mathfrak{z}})\right) \leq -\inf_{\boldsymbol{\rho} \in F}   J(\boldsymbol{\rho}).\]

 \item[(ii)]
 For every open set $G\subset \mathcal P_S$, 
we have
 \[ \liminf_{n\to \infty} \frac{1}{n} \log  \E_s \left( \sum_{\mathfrak{z}\in \mathcal B_n} \indset{G}(L_{\mathfrak{z}})\right) \geq -\inf_{\boldsymbol{\rho} \in G}   J(\boldsymbol{\rho}).\]

\end{enumerate}

\end{lemma}

Our main result in this section implies that on average,
 the population model with ancestral memory grows  faster than that without memory, even when one only considers individuals with given  empirical types distribution. 
A related observation has been made for a different random evolution with memory in \cite[Theorem 4.2]{BMin}.

\begin{theorem} \label{T3} There is the inequality \[\Lambda^*\leq J.\]
\end{theorem} 
\begin{proof} For every $\blambda\in \R^S$, consider the biased probability measure 
$\Sigma_{\e^{\blambda}} \in \mathcal P_S$ given by 
\[\Sigma_{\e^{\blambda}}(s)= \frac{\sigma(s) \e^{\lambda(s)} }{\Sigma \e^{\blambda}}, \qquad s\in S.\]
Then note  from Jensen's inequality,  that for every $y\geq 0$, we have 
\[\Sigma \exp(\blambda+ y\m) =\Sigma \e^{\blambda} \times \Sigma_{\e^{\blambda}} \exp(y\m) 
\geq \Sigma \e^{\blambda} \times  \exp( y  \Sigma_{\e^{\blambda}}  \m),\]
so\[  \int_0^{\infty} \frac{\dd y}{\Sigma \exp(\blambda+ y\m)}\ \leq  \frac{1}{\Sigma \e^{\blambda} \times   \Sigma_{\e^{\blambda}}  \m}=
\frac{1}{\Sigma(\e^{\blambda}  \m)}.
\]
We thus have
\[ \Lambda( \blambda) \geq
 \log (\Sigma(\e^{\blambda}  \m)) = \log (\Sigma_{\m}(\e^{\blambda} )) + \log \Sigma \m.
 \]
 The claim now follows from the well-known fact that the Kullback--Leibler divergence coincides with the Fenchel--Legendre transform of the cumulant generating function.
\end{proof}

\bibliography{solvbm.bib}

\end{document}